\documentclass[12pt, twoside]{article}
\usepackage{amsmath,amsthm,amssymb}
\usepackage{times}
\usepackage{enumerate}
\usepackage{mathrsfs}
\usepackage[all]{xy}

\def\titlerunning#1{\gdef\titrun{#1}}
\makeatletter
\def\author#1{\gdef\autrun{\def\and{\unskip, }#1}\gdef\@author{#1}}
\def\address#1{{\def\and{\\\hspace*{18pt}}\renewcommand{\thefootnote}{}%
\footnote {#1}}%
\markboth{\autrun}{\titrun}}
\makeatother
\def\email#1{e-mail: #1}
\def\subjclass#1{{\renewcommand{\thefootnote}{}%
\footnote{\emph{Mathematics Subject Classification (2010):} #1}}}
\def\keywords#1{\par\medskip
\noindent\textbf{Keywords.} #1}

\newtheorem{thm}{Theorem}[section]
\newtheorem{cor}[thm]{Corollary}

\newtheorem{prop}[thm]{Proposition}

\theoremstyle{definition}

\newtheorem{rem}[thm]{Remark}

\numberwithin{equation}{section}

\begin{document}

%%%%% To ease editing, add:

\baselineskip=17pt

%%%%%%%%%%%%%%%%

%% In the running head, give an abbreviation of the title.
\titlerunning{}

\title{Scalar curvatures in almost Hermitian geometry and some applications}

\author{Jixiang Fu
\and
Xianchao Zhou}

\date{}

\maketitle

\address{J. Fu: School of Mathematical Sciences, Fudan University, Shanghai 200433, China; \email{majxfu@fudan.edu.cn}
\and
X. Zhou: Department of Applied Mathematics, Zhejiang University of Technology, Hangzhou 310023, China; \email{zhouxianch07@zjut.edu.cn}}


%%%%%%%%

\begin{abstract}
On an almost Hermitian manifold, we have two Hermitian scalar curvatures with respect to any canonical Hermitian connection defined
by P. Gauduchon. Explicit formulas of these two
Hermitian scalar curvatures are obtained in terms of Riemannian scalar curvature, norms of decompositions of covariant derivative of the fundamental 2-form
with respect to the Levi-Civita connection, and the codifferential of the Lee form. Then we get some inequalities of various total scalar curvatures and some
characterization results of the K\"{a}hler metric, balanced metric, locally conformally  K\"{a}hler metric and the k-Gauduchon metric.
As corollaries, we show some results related to a problem given by Lejmi-Upmeier \cite{LeU} and a conjecture given
by Angella-Otal-Ugarte-Villacampa \cite{AOUV}.

%% Keywords are optional
\end{abstract}

\section{Introduction}

In K\"{a}hler geometry, the complex structure is parallel with respect to the Levi-Civita connection.
There is a well connection between the complex geometry and the underlying Riemannian geometry
on a K\"{a}hler manifold. S. T. Yau \cite{Yau} proved that the Kodaira dimension of a compact  K\"{a}hler manifold with positive total
 scalar curvature must be $-\infty$. He \cite{Yau} also proved that a
compact K\"{a}hler surface is uniruled if and only if it admits a K\"{a}hler metric with positive total scalar curvature.
For a Hermitian non-K\"{a}hler manifold, the complex structure is not parallel with respect to the Levi-Civita connection. We usually choose the Chern connection instead of the Levi-Civita connection and hence we have the  (total) Chern scalar curvature.
I. Chiose, R. R\u{a}sdeaconu and I. \c{S}uvaina \cite{CRS} then successfully extended Yau's result to the non-K\"ahler case and showed that
a Moishezon manifold is uniruled if and only if it admits a balanced metric of positive total Chern scalar curvature. Recently, X. K. Yang \cite{Yang} have proved that a compact complex manifold $M$ admits a Hermitian metric with positive (resp., negative) Chern scalar curvature if and only if the canonical line bundle $K_M$ (resp., anti-canonical line bundle $ K_M^{\ast}$) is not pseudo-effective.

The purpose of this paper is to study various (total) scalar curvatures on an almost Hermitian manifold. We first recall some definitions given by P. Gauduchon \cite{Gau3}. Let $(M,J,h)$ be an almost Hermitian manifold of real dimension $2n$.
Let $\nabla$ be the Levi-Civita connection of $h$ and $F$ the associated fundamental 2-form.
P. Gauduchon \cite{Gau3} introduced a 1-parameter family $D^t$ of canonical Hermitian connections as follows:
\begin{equation}
\begin{aligned}
h(D_X^t Y,Z)
=&h(\nabla_X Y-\frac{1}{2}J(\nabla_X J)Y,Z)\\
+&\frac{t}{4}[h((\nabla_{JY}J) Z+J(\nabla_Y J)Z,X)-h((\nabla_{JZ}J) Y+J(\nabla_Z J)Y,X)].\notag
\end{aligned}
\end{equation}
There are three important cases: $D^0$ is the first canonical Hermitian connection, also called the Lichnerowicz connection \cite{Kob2};  $D^1$ is the second canonical Hermitian connection, also called the Chern connection because  it coincides
with the connection used by S. S. Chern \cite{Chern} in the integrable case; $D^{-1}$ is the Bismut connection. In the integrable case, $D^{-1}$ is characterized by its torsion being skew-symmetric \cite{Bis}.

On an almost Hermitian manifold, we have the Riemannian
scalar curvature $s$ and the $J$-scalar curvature $s_J$ associated to the Levi-Civita connection.  The relations and some applications of these two scalar curvatures are discussed in  section 2. Meanwhile, for any canonical Hermitian connection $D^t$, by using contractions of  the corresponding curvature tensor
$K^t$, two Hermitian scalar curvatures $s_1(t)$ and $s_2(t)$  can be defined. In fact, for a given unitary frame field $\{u_i\}_{i=1,2,...,n}$, we define
\begin{equation*}
s_1(t)=K^t(u_{\bar i}, u_i,u_j,u_{\bar j})\quad  \textup{and}\quad  s_2(t)=K^t(u_{\bar i},u_j,u_i,u_{\bar j}).
\end{equation*}
For a Hermitian manifold with the Chern connection $D^1$ or the Bismut connection $D^{-1}$, the relations between the corresponding Hermitian scalar curvatures and the Riemannian scalar curvature have been presented in many papers  \cite{Gau1,AlI,IP1,LY}.
In this paper, we mainly establish the following two identities.

\begin{thm}\textup{(=Theorem 4.3.)}
Let $(M,J,h)$ be an almost Hermitian manifold of real dimension $2n$.
Then
\begin{equation}
\begin{aligned}
s_1(t)
=&\frac{s}{2}-\frac{5}{12}|(dF)^-|^2+\frac{1}{16}|N^0|^2+\frac{1}{4}|(dF)_0^+|^2\notag\\
&+[\frac{1}{4(n-1)}
+\frac{t-1}{2}]|\alpha_F|^2+\frac{t-2}{2}\delta \alpha_F
\end{aligned}
\end{equation}
and
\begin{equation}
\begin{aligned}
s_2(t)
=&\frac{s}{2}-\frac{1}{12}|(dF)^-|^2+\frac{1}{32}|N^0|^2-\frac{t^2-2t}{4}|(dF)_0^+|^2\notag\\
&-[\frac{t^2-2t}{4(n-1)}+\frac{(t+1)^2}{8}]|\alpha_F|^2-\frac{t+1}{2}\delta \alpha_F.
\end{aligned}
\end{equation}
\end{thm}
Here, $(dF)^-$, $N^0$, $(dF)_0^+$ and $\alpha_F$ are the four components of $\nabla F$ \cite{Gau3}. By letting some of them equal zero,
A. Gray and L. M. Hervella \cite{GrH} defined 16 classes (4 classes for $n=2$) of almost Hermitian manifolds. We will use the same notations
of classes of almost Hermitian manifolds as in \cite{GrH}. For example, the class $\mathcal{W}_1\oplus \mathcal{W}_4$: $N^0=(dF)_0^+=0$;
the class $\mathcal{W}_3\oplus \mathcal{W}_4$: $(dF)^-=N^0=0$; the class $\mathcal{W}_2\oplus\mathcal{W}_3\oplus \mathcal{W}_4$: $(dF)^-=0$.

In the following, we give some applications of Theorem 1.1.

\begin{thm}\textup{(=Theorem 5.1.)}
Let $(M,J,h)$ be a compact almost Hermitian manifold of real dimension $2n$ ($n\geq3$).

(1) If $(M,J,h)\in \mathcal{W}_2\oplus \mathcal{W}_3\oplus \mathcal{W}_4$
and $t\geq 1-\frac{1}{2(n-1)}$, then
\begin{equation*}
\int_M [2s_1(t)-s]dv\geq 0.
\end{equation*}
The equality holds if and only if $(M,J,h)$ is a locally conformally  K\"{a}hler manifold when $t=1-\frac{1}{2(n-1)}$ or
$(M,J,h)$ is a K\"{a}hler manifold when $t>1-\frac{1}{2(n-1)}$.

(2) If $(M,J,h)\in \mathcal{W}_1\oplus \mathcal{W}_4$
and $t\leq 1-\frac{1}{2(n-1)}$, then
$$\int_M [2s_1(t)-s]dv\leq 0.$$
The equality holds  if and only if $(M,J,h)$ is a locally conformally  K\"{a}hler manifold when $t=1-\frac{1}{2(n-1)}$ or
$(M,J,h)$ is a K\"{a}hler manifold when $t<1-\frac{1}{2(n-1)}$.
\end{thm}

Recently, for the Chern connection $D^1$,  M. Lejmi and  M. Upmeier have given a problem in  Remark 3.3 in \cite{LeU}:  do higher-dimensional compact almost Hermitian non-K\"{a}hler manifolds with $2s_1(1)=s$ exist? For this question, from Theorem 1.2 (1), we have the following non-existence result.

\begin{cor}\textup{(=Corollary 5.3.)}
Let $(M,J,h)$ be a compact almost Hermitian manifold of real dimension $2n$ ($n\geq3$).  If $(M,J,h)\in \mathcal{W}_2\oplus \mathcal{W}_3
\oplus \mathcal{W}_4$ and $2s_1(1)=s$, then $(M,J,h)$ is a K\"{a}hler manifold.\\
\end{cor}

\begin{thm}\textup{(=Theorem 5.4.)}
Let $(M,J,h)$ be a compact almost Hermitian manifold of real dimension $2n$ ($n\geq 3$).

(1) If $(M,J,h)\in \mathcal{W}_2\oplus \mathcal{W}_3\oplus \mathcal{W}_4$, $t\in (-\infty, -3-2\sqrt{3}]\cup [-3+2\sqrt{3},+\infty)$, then
$$\int_M[s_1(t)-s_2(t)]dv\geq 0.$$
The equality holds if and only if $(M,J,h)$ is a balanced Hermitian manifold when $t=1$ or  $(M,J,h)$ is a K\"{a}hler manifold when $t\neq 1$.

(2) If $(M,J,h)\in \mathcal{W}_1\oplus \mathcal{W}_4$, $t\in [-1,\frac{1}{3}]$, then
$$\int_M[s_1(t)-s_2(t)]dv\leq 0.$$
The equality holds if and only if  $(M,J,h)$ is a K\"{a}hler manifold.
\end{thm}

Recently, B.  Yang and F. Y. Zheng \cite{YZ1}, and D. Angella, A. Otal, L. Ugarte and R. Villacampa \cite{AOUV}
have introduced the definition of K\"{a}hler-like Hermitian structure on a Hermitian manifold.  As a corollary of Theorem 1.4 (1),  we prove conjecture 2 in \cite{AOUV} for $t\in (-\infty, -3-2\sqrt{3}]\cup [-3+2\sqrt{3},1)\cup(1,+\infty)$.

\begin{cor}\textup{(=Corollary 5.7.)}
Let $(M,J,h)$ be a compact Hermitian manifold of real dimension $2n$ ($n\geq 3$).  If $(M,J,h)$ is K\"{a}hler-like for a canonical Hermitian connection $D^t$ with $t\in (-\infty, -3-2\sqrt{3}]\cup [-3+2\sqrt{3},1)\cup(1,+\infty)$, then $(M,J,h)$ is a K\"{a}hler manifold.
\end{cor}

Although $t=-1$ (i.e., the Bismut connection) is not contained in Corollary 1.5, we have

\begin{thm}\textup{(=Theorem 5.8.)}
Let $(M,J,h)$ be a compact Hermitian manifold of real dimension $2n$ ($n\geq 3$). Then
$$\int_M[s_1(-1)-s_2(-1)]dv=\int_M(|dF|^2-|\alpha_F|^2)dv.$$
In particular, if $h$ is a Gauduchon metric (i.e., $\delta\alpha_F=0$)  and $s_1(-1)=s_2(-1)$, then
$h$ is a $k$-Gauduchon metric \cite{FWW}, that is
$$\sqrt{-1}\partial\bar{\partial} (F^k)\wedge F^{n-k-1}=0,$$ for $k=1,2,...,n-1$.
\end{thm}

The paper is organized as follows. In section 2,  we provide a brief description of the almost Hermitian geometry. An
important  formula between the Riemannian scalar curvature and  the $J$-scalar curvature is obtained.
In section 3, we show the explicit formulas of the two Hermitian scalar curvatures of the Lichnerowicz connection on an almost Hermitian manifold.
In section 4, by using the structure equations of the Lichnerowicz connection $D^0$
and the Chern connection $D^1$, we get the corresponding curvature formulas of Gauduchon's family of canonical Hermitian connections $D^t$.
In section 5, we show some applications of Theorem 1.1.

For convenience, we fix the index range $1\leq i,j,k,...\leq n$, $1\leq A,B,C,...\leq 2n$. We should use the Einstein summation convention, i.e., repeated indices are summed over. The space of smooth tangent vector fields on $M$ is denoted by $\Gamma(TM)$.

\section{J-scalar curvature of an almost Hermitian manifold}

In this section, we provide a brief description of the almost Hermitian geometry. In particular, we review the decomposition
of the covariant derivative of the fundamental 2-form with respect to the Levi-Civita connection.
Then we introduce an important  formula between the Riemannian
scalar curvature and  the $J$-scalar curvature. Some applications are also discussed.

Let $(M,J,h)$ be an almost Hermitian manifold of real dimension $2n$. $J$ is an almost complex structure which is orthogonal with respect to
the Riemannian metric $h=\langle, \rangle$. The Levi-Civita connection of $h$ is denoted by $\nabla$. The Nijenhuis tensor $N$ is defined by
$$N(X,Y)=[X,Y]+J[JX,Y]+J[X,JY]-[JX,JY],$$
where $X,Y\in\Gamma(TM)$. It is well-known that $N\equiv 0$ if and only if $J$ is integrable \cite{NeN}. The associated fundamental 2-form $F$ is defined by
$$F(X,Y)=h(JX,Y),$$
and the volume form is denoted by $dv=\frac{F^n}{n!}$.

Another important differential form of $(M,J,h)$ is the Lee form $\alpha_F$, defined by $\alpha_F=J\delta F$, where $\delta=-\ast d \ast$ is the codifferential
with respect to $h$. In fact, the Lee form $\alpha_F$ is also determined by
\begin{equation}
dF=(dF)_0+\frac{1}{n-1}\alpha_F\wedge F,
\end{equation}
where $(dF)_0$ denotes the primitive part of $dF$.

The covariant derivative of $F$ with respect to the Levi-Civita connection $\nabla$ is
\begin{equation}
(\nabla_X F)(Y,Z)=\frac{1}{2}[dF(X,Y,Z)-dF(X,JY,JZ)-\langle JX,N(Y,Z) \rangle].
\end{equation}
$\nabla F$ also satisfies
\begin{equation}
(\nabla_X F)(Y,Z)=-(\nabla_X F)(Z,Y)=-(\nabla_X F)(JY,JZ),
\end{equation}
where $X,Y,Z\in\Gamma(TM)$.

According to Proposition 1 in \cite{Gau3}, $\nabla F$ has the following decomposition,
\begin{eqnarray}
&&(\nabla_X F)(Y,Z)
=(dF)^-(X,Y,Z)-\frac{1}{2}N(JX,Y,Z)\notag\\
&&~~~~~~~~~~~~~~~~~~~~~~~~~~~~+\frac{1}{2}[(dF)^+(X,Y,Z)-(dF)^+(X,JY,JZ)],
\end{eqnarray}
where $X,Y,Z\in \Gamma(TM)$, $N(JX,Y,Z)=\langle JX,N(Y,Z)\rangle$, $(dF)^+$ is the $(2,1)+(1,2)$-component of $dF$,  $(dF)^-$ is the $(3,0)+(0,3)$-component of $dF$. Furthermore,
\begin{equation}
N=N^0+\mathfrak{b}N,
\end{equation}
\begin{equation}
(dF)^+=(dF)_0^+ +\frac{1}{n-1}\alpha_F\wedge F,
\end{equation}
where $\mathfrak{b}N$ is the skew-symmetric part of $N$, defined by
$\mathfrak{b}N(X,Y,Z)=\frac{1}{3}[N(X,Y,Z)+N(Y,Z,X)+N(Z,X,Y)]$,  $N^0$ denotes the component of the Nijenhuis tensor satisfying $\mathfrak{b}N^0=0$, $(dF)_0^+$ is the primitive part of $(dF)^+$.

For an almost Hermitian manifold, the four components $(dF)^-$, $N^0$, $(dF)_0^+$ and $\alpha_F$ carry important geometric information. Each of the 16 classes
( 4 classes for $n=2$) of almost Hermitian manifolds given by the Gray-Hervella classification \cite{GrH}  corresponds to the vanishing of some subset of
$\{(dF)^-, N^0, (dF)_0^+, \alpha_F\}$. For example, $\mathcal{W}_1$= the class of nearly K\"{a}hler manifolds: $N^0=(dF)_0^+=\alpha_F=0$;
$\mathcal{W}_2$= the class of almost K\"{a}hler manifolds: $(dF)^-=(dF)_0^+=\alpha_F=0$, equivalently, $dF=0$;
$\mathcal{W}_3$= the class of semi-K\"{a}hler (or balanced \cite{Mic}) Hermitian manifolds: $(dF)^-=N^0=\alpha_F=0$;
$\mathcal{W}_3\oplus \mathcal{W}_4$= the class of Hermitian manifolds: $(dF)^-=N^0=0$.

The metric $h$ induces natural inner product, also denoted by $\langle, \rangle$, on $\wedge^k M$ (the bundle of real $k$-forms) and  on
$TM\otimes \wedge^k M$ (the bundle of $TM$-valued $k$-forms). From the decomposition (2.4), then the norm formula of $\nabla F$ is \cite{Gau3}
\begin{align}\label{E:1}
|\nabla F|^2 &= |dF|^2+\frac{1}{4}|N^0|^2-\frac{2}{3}|(dF)^-|^2\notag\\
             &= |(dF)^+|^2+\frac{1}{4}|N^0|^2+\frac{1}{3}|(dF)^-|^2\notag\\
             &=\frac{|\alpha_F|^2}{n-1}+|(dF)_0^+|^2+\frac{1}{4}|N^0|^2+\frac{1}{3}|(dF)^-|^2.
\end{align}
In particular, if $J$ is integrable, then $|\nabla F|^2=|dF|^2$.

The Riemannian curvature tensor $R$ is defined by
\begin{equation}
R(X,Y,Z,W)=\langle\nabla_Z \nabla_W Y-\nabla_W \nabla_Z Y-\nabla_{[Z,W]}Y, X\rangle,
\end{equation}
where $X,Y,Z,W\in \Gamma(TM)$. Meanwhile, the Riemannian curvature tensor $R$ induces an
important curvature operator $\mathfrak{R}:\Lambda^2 TM\longrightarrow \Lambda^2 TM$, defined by
\begin{equation}
h(\mathfrak{R}(X\wedge Y), Z\wedge W)=R(X,Y,Z,W).
\end{equation}
One can also use the metric $h$ to view the curvature operator $\mathfrak{R}$ as an endomorphism of $\wedge^2M$.

Let $\{e_1,e_2,...,e_{2n}\}$ be a local orthonormal frame field of $(M, h)$. The Ricci tensor is
$Ric(X,Y)=R(e_A,X,e_A,Y)$, the corresponding Riemannian scalar curvature is $s=Ric(e_A,e_A)$.

Indeed, on an almost Hermitian manifold $(M,J,h)$, there is a $J$-twisted version of the Ricci tensor, called the $J$-Ricci tensor from now on (in some literatures, also called the $\ast$-Ricci tensor) \cite{GrH,TrV,dRS}. The $J$-Ricci tensor, denoted by $Ric_{J}$, is defined by
$$Ric_{J}(X,Y)=R(e_A,X,Je_A,JY).$$
The corresponding $J$-scalar curvature, denoted by $s_J$, is given by $s_J=Ric_J (e_A,e_A)$. In general,
$Ric_J\neq Ric$, $Ric_J (X,Y)\neq Ric_J(Y,X)$, but $Ric_J (X,Y)=Ric_J(JY,JX)$. Thus, we can introduce the following 2-form $\rho_J$, called the
$J$-Ricci form,
\begin{equation}
\rho_J(X,Y)=-Ric_J(X,JY).
\end{equation}
By direct calculations, we have
\begin{equation}
\rho_J=\mathfrak{R}(F),~~~~~s_J=2\langle\mathfrak{R}(F),F\rangle.
\end{equation}

Let $\Delta_d=d\delta+\delta d$ be the Hodge Laplacian operator associated to the metric $h$. By using the Bochner-Weitzenb\"{o}ck formula for the fundamental 2-form
$F$, we obtain \cite{Gau2}
\begin{eqnarray}
&&\Delta_d F
=\nabla^\ast\nabla F+\frac{2(n-1)}{n(2n-1)}s\cdot F\notag\\
&&~~~~~~~~~~~~~-2 W(F)+\frac{n-2}{n-1}[Ric_0(J\cdot,\cdot)-Ric_0(\cdot,J\cdot)],
\end{eqnarray}
where $\nabla^\ast$ denotes the adjoint of the Levi-Civita connection $\nabla$ with respect to $h$, $Ric_0$ is the
trace-free part of the Ricci tensor $Ric$, and $W$ is the Weyl curvature operator.

By contracting (2.12) by $F$, then
\begin{equation}
\langle\Delta_d F, F\rangle-\langle\nabla^\ast\nabla F,F\rangle=\frac{2(n-1)}{2n-1}s-2\langle W(F),F\rangle.
\end{equation}

Since
$$\langle\Delta_d  F,F\rangle=|dF|^2+|\alpha_F|^2+2\delta\alpha_F$$
and $\langle\nabla^\ast\nabla F, F\rangle=|\nabla F|^2$, it follows that
\begin{equation}
\frac{2(n-1)}{2n-1}s-2\langle W(F),F\rangle=|dF|^2+|\alpha_F|^2+2\delta\alpha_F-|\nabla F|^2.
\end{equation}
In particular, if $(M,J,h)$ is conformally flat, i.e., $W=0$, then
\begin{equation}
s=\frac{2n-1}{2(n-1)}[|dF|^2+|\alpha_F|^2+2\delta\alpha_F-|\nabla F|^2].
\end{equation}

In general, from the decomposition of the Riemannian curvature tensor \cite{Bes,dRS}, we have
\begin{equation}
\langle W(F),F\rangle=\frac{1}{2(2n-1)}[(2n-1)s_J-s].
\end{equation}
Combing (2.14), (2.16) and (2.7), we obtain the following difference between the Riemannian scalar curvature $s$ and $J$-scalar curvature $s_J$,
\begin{align}\label{E:1}
s-s_J &=|dF|^2-|\nabla F|^2+|\alpha_F|^2+2\delta\alpha_F\notag\\
             &=\frac{2}{3}|(dF)^-|^2-\frac{1}{4}|N^0|^2+|\alpha_F|^2+2\delta\alpha_F.
\end{align}
We remark that from (2.7) and (2.17), we can prove all the linear relations in Theorem 7.2 in \cite{GrH}. Using the formula (2.17), we have

\begin{thm} Let $(M,J,h)$ be a compact almost Hermitian manifold of real dimension $2n$.
If $(M,J,h)\in\mathcal{W}_1\oplus\mathcal{W}_3\oplus \mathcal{W}_4$, then
$$\int_M(s-s_J)dv\geq 0.$$
The equality holds if and only if $(M,J,h)$ is a balanced Hermitian manifold.
\end{thm}

\begin{proof} From the decomposition of $\nabla F$, if $(M,J,h)\in\mathcal{W}_1\oplus\mathcal{W}_3\oplus \mathcal{W}_4$,
then the component $N^0=0$. From (2.17), we have
$$s-s_J=\frac{2}{3}|(dF)^-|^2+|\alpha_F|^2+2\delta\alpha_F,$$
which implies
$$\int_M(s-s_J)dv=\frac{2}{3}\int_M|(dF)^-|^2dv+\int_M|\alpha_F|^2dv\geq 0.$$
It is obvious that the above equality holds if and only if $(dF)^-=\alpha_F=0$. Together with $N^0=0$, the result follows.

\end{proof}

\begin{rem}
Recently, B. Yang and F. Y. Zheng \cite{YZ1} have proved that for a compact Hermitian manifold $(M,J,h)$ with Riemannian curvature tensor $R$
satisfying the Gray-K\"{a}hler-like condition \cite{Gra} (i.e., $R(X,Y,Z,W)=R(X,Y,JZ,JW)$), then $(M,J,h)$ must be a balanced Hermitian manifold. This result
is included in our Theorem 2.1. Because the Gray-K\"{a}hler-like condition implies $s-s_J=0$, and moreover, our condition $(M,J,h)\in\mathcal{W}_1\oplus\mathcal{W}_3\oplus \mathcal{W}_4$, not necessary a Hermitian manifold.
\end{rem}

By using the same method as in Theorem 2.1, we obtain

\begin{thm} Let $(M,J,h)$ be a compact almost Hermitian manifold of real dimension $2n$.
If $(M,J,h)\in\mathcal{W}_2\oplus\mathcal{W}_3$, then
$$\int_M(s-s_J)dv\leq0.$$
The equality holds if and only if $(M,J,h)$ is a balanced Hermitian manifold.
\end{thm}
\begin{proof} From the decomposition of $\nabla F$, if $(M,J,h)\in\mathcal{W}_2\oplus\mathcal{W}_3$,
then the component $(dF)^-=\alpha_F=0$. From (2.17), we have
$$s-s_J=-\frac{1}{4}|N^0|^2=-\frac{1}{4}|N|^2,$$
which implies
$$\int_M(s-s_J)dv=-\frac{1}{4}\int_M|N|^2dv\leq 0.$$
Thus, the theorem follows.
\end{proof}

\begin{rem} We should remark some interesting observations. From Theorem 2.1, it follows that
compact real hyperbolic manifold (of dimension $2n\geq 4$) does not admit a compatible $\mathcal{W}_1\oplus\mathcal{W}_3\oplus \mathcal{W}_4$-structure, in particular, does not admit an orthogonal complex structure, since real hyperbolic manifold is conformally flat with $s<0$ \cite{HeL}\cite{Gau2}.
Theorem 2.3 implies that the six sphere $S^6$ with the standard round metric does not admit a compatible $\mathcal{W}_2\oplus\mathcal{W}_3$-structure, since $S^6$ with the standard round metric is conformally flat with $s>0$ \cite{HeL}. By the same reasons, the Hopf manifold $S^1\times S^{2n-1}$ with the standard induced metric does not admit a compatible $\mathcal{W}_2\oplus\mathcal{W}_3$-structure.

\end{rem}

\section{Scalar curvatures of the Lichnerowicz connection}

In this section, we show the explicit formulas of the two Hermitian scalar curvatures of the Lichnerowicz connection on an almost Hermitian manifold.
Preparing for the next section, the corresponding structure equations of the Lichnerowicz connection are also presented.

Let $(M^{2n},J,h)$ be an almost Hermitian manifold with Levi-Civita connection $\nabla$. The Lichnerowicz connection $D^0$ is defined by
\begin{equation}
D_X^0 Y=\nabla_X Y-\frac{1}{2}J(\nabla_X J)Y,
\end{equation}
and the corresponding curvature tensor $K^0$ is
\begin{equation}
K^0(X,Y,Z,W)=\langle D_Z^0 D_W^0 Y-D_W^0 D_Z^0 Y-D_{[Z,W]}^0 Y, X\rangle,
\end{equation}
where $X,Y,Z,W\in \Gamma(TM)$.

A straightforward computation shows the following relation of the curvature tensors $K^0$ and $R$ \cite{GBNV,DGM},
\begin{eqnarray}
&&K^0(X,Y,Z,W)
=\frac{1}{2}[R(X,Y,Z,W)+R(JX,JY,Z,W)]\\
&&~~~~~~~~~~~~~~~~~~~~~~~~~~~~~~+\frac{1}{4}[\langle(\nabla_Z J)X, (\nabla_W J)Y\rangle-\langle(\nabla_W J)X, (\nabla_Z J)Y\rangle].\notag
\end{eqnarray}

From the above formula, as in the Hermitian case \cite{Gau1,LY}, we can define various Ricci forms and Hermitian scalar curvatures by some usful contractions of the curvature tensor $K^0$. Here, we only consider two Hermitian scalar curvatures, denoted by $s_1(0)$ and $s_2(0)$, of the Lichnerowicz connection.

With respect to the almost Hermitian structure $(h, J)$, we shall always choose a local $J$-adapted orthonormal frame field
$\{e_i, e_{n+1}=Je_i\}_{i=1,2,...,n}$. Then the corresponding unitary frame field is
$\{u_i=\frac{1}{\sqrt{2}}(e_i-\sqrt{-1}e_{n+i})\}_{i=1,2,...,n}$. Set $u_{\bar{i}}=\overline{u_i}$. The first Hermitian scalar curvature $s_1(0)$ and
the second Hermitian scalar curvature $s_2(0)$ are defined by
$$s_1(0)=K^0(u_{\bar{i}},u_i,u_j,u_{\bar{j}}),~~~~s_2(0)=K^0(u_{\bar{i}},u_j,u_i,u_{\bar{j}}).$$

From the relation (3.3) and (2.11), we have
\begin{align}\label{E:1}
s_1(0) &=R(u_{\bar{i}},u_i,u_j,u_{\bar{j}})+\frac{1}{4}[\langle(\nabla_{u_j}J)u_{\bar{i}},(\nabla_{u_{\bar{j}}}J)u_i\rangle
-\langle(\nabla_{u_{\bar{j}}}J)u_{\bar{i}},(\nabla_{u_j}J)u_i\rangle]\notag\\
             &=\frac{1}{2}s_J-\frac{1}{8}\langle(\nabla_{e_B}J)(Je_A),(\nabla_{Je_B}J)e_A\rangle.
\end{align}
According to the decomposition (2.4) of $\nabla F$, it follows that
\begin{align}\label{E:1}
\langle(\nabla_{e_B}J)(Je_A),(\nabla_{Je_B}J)e_A\rangle &=-\langle(\nabla_{Je_B}J)(Je_A),(\nabla_{e_B}J)e_A\rangle \notag\\
             &=-2|(dF)^+|^2+\frac{2}{3}|(dF)^-|^2+\frac{1}{2}|N^0|^2.
\end{align}
Then, from (3.4) and (3.5), we get
\begin{equation}
s_1(0)=\frac{1}{2}s_J+\frac{1}{4}|(dF)^+|^2-\frac{1}{12}|(dF)^-|^2-\frac{1}{16}|N^0|^2.
\end{equation}

For $s_2(0)$, by the similar method, we have
\begin{equation}
s_2(0)=R(u_{\bar{i}},u_j,u_i,u_{\bar{j}})+
\frac{1}{4}[\langle(\nabla_{u_i}J)u_{\bar{i}},(\nabla_{u_{\bar{j}}}J)u_j\rangle-\langle(\nabla_{u_{\bar{j}}}J)u_{\bar{i}},(\nabla_{u_i}J)u_j\rangle].
\end{equation}
Using the first Bianchi identity, it follows
\begin{align}\label{E:1}
R(u_{\bar{i}},u_j,u_i,u_{\bar{j}}) &=R(u_{\bar{i}},u_i,u_j,u_{\bar{j}})+R(u_{\bar{i}},u_{\bar{j}},u_i,u_j) \notag\\
             &=\frac{1}{2}R(u_{\bar{i}},u_i,u_j,u_{\bar{j}})+\frac{1}{4}[4R(u_{\bar{i}},u_{\bar{j}},u_i,u_j)+2R(u_{\bar{i}},u_i,u_j,u_{\bar{j}})]\notag\\
             &=\frac{1}{4}s_J+\frac{1}{4}s.
\end{align}
Note that
\begin{equation}
\langle(\nabla_{u_i}J)u_{\bar{i}},(\nabla_{u_{\bar{j}}}J)u_j\rangle=\frac{1}{2}\langle(\nabla_{e_A}J)e_A,(\nabla_{e_B}J)e_B\rangle=\frac{1}{2}|\alpha_F|^2,
\end{equation}
and
\begin{align}\label{E:1}
\langle(\nabla_{u_{\bar{j}}}J)u_{\bar{i}},(\nabla_{u_i}J)u_j\rangle&=\frac{1}{2}\langle(\nabla_{e_A}J)e_B,(\nabla_{e_B}J)e_A\rangle\notag\\
&=\frac{1}{2}(|\nabla F|^2-|dF|^2).
\end{align}
Then, from (3.7), (3.8), (3.9) and (3.10), we obtain
\begin{equation}
s_2(0)=\frac{1}{4}(s_J+s)+\frac{1}{8}[|\alpha_F|^2-(|\nabla F|^2-|dF|^2)].
\end{equation}

Combing (2.7) and (2.17), the formulas (3.6) and (3.11) yield the following theorem.

\begin{thm}
Let $(M,J,h)$ be an almost Hermitian manifold of real dimension $2n$.
Then
\begin{equation}
s_1(0)=\frac{s}{2}-\frac{5}{12}|(dF)^-|^2+\frac{1}{16}|N^0|^2+\frac{1}{4}|(dF)_0^+|^2+\frac{3-2n}{4(n-1)}
|\alpha_F|^2-\delta \alpha_F,
\end{equation}
and
\begin{equation}
s_2(0)=\frac{s}{2}-\frac{1}{12}|(dF)^-|^2+\frac{1}{32}|N^0|^2
-\frac{1}{8}|\alpha_F|^2-\frac{1}{2}\delta \alpha_F.
\end{equation}

\end{thm}

\begin{rem}
The integral formula of (3.12) is  given in \cite{BH2}. This formula is very useful in characterizing different types of almost Hermitian structures, in
particular, for these with vanishing first Chern class \cite{SeV,BH2}.
\end{rem}

Note that for other canonical Hermitian connections  mentioned in the introduction, the corresponding curvature tensor is extremely complicated,
not so simple as in the formula (3.3).  In the next section, we should exploit the moving frame method, which turns out
 to be very effective in our study of curvatures in almost Hermitian geometry. At the end of this section,
 we also show the structure equations of the Lichnerowicz connection on an almost Hermitian manifold, preparing for the next section.

Let $\{e_i, e_{n+1}=Je_i\}_{i=1,2,...,n}$ be a local $J$-adapted orthonormal frame field on $(M,J,h)$, its dual coframe field is denoted by
$\{\omega^1,\omega^2,...,\omega^{2n}\}$. Let $\omega=(\omega_B^A)$ be the connection form matrix of the Levi-Civita connection. The corresponding curvature
form matrix is denoted by $\Omega=(\Omega_B^A)$. Thus, the structure equations of the Levi-Civita connection are
\begin{eqnarray}
&&d\omega^A=-\omega_B^A\wedge\omega^B,\\
&&d\omega_B^A=-\omega_C^A\wedge\omega_B^C+\Omega_B^A,
\end{eqnarray}
where $\omega_B^A+\omega_A^B=0$, $\Omega_B^A=\frac{1}{2}R_{ABCD}\omega^C\wedge\omega^D$,
$R_{ABCD}=R(e_A,e_B,e_C,e_D)$.

Set $J_0=\left(\begin{array}{cc}0 &-I_n\\I_n&0
\end{array}\right)$, $I_n$ is the $n\times n$ identity matrix. Then we have the following decompositions of $\omega$ and $\Omega$,
$$\omega=\frac{1}{2}(\omega-J_0 \omega J_0)+\frac{1}{2}(\omega+J_0 \omega J_0),$$
$$\Omega=\frac{1}{2}(\Omega-J_0 \Omega J_0)+\frac{1}{2}(\Omega+J_0 \Omega J_0).$$
In fact, $\frac{1}{2}(\omega-J_0 \omega J_0)$ is the connection form matrix of the Lichnerowicz connection.
The unitary coframe field is denoted by $\theta^i=\frac{1}{\sqrt{2}}(\omega^i+\sqrt{-1}\omega^{n+i})$, and set $$\varphi_j^i=\frac{1}{2}(\omega_j^i+\omega_{n+j}^{n+i})+\frac{\sqrt{-1}}{2}(\omega_j^{n+i}-\omega_{n+j}^i),$$
$$\mu_j^i=\frac{1}{2}(\omega_j^i-\omega_{n+j}^{n+i})+\frac{\sqrt{-1}}{2}(\omega_j^{n+i}+\omega_{n+j}^i).$$
Obviously, $\varphi_j^i+\overline{\varphi_i^j}=0$ and  $\mu_j^i+\mu_i^j=0$.
Then the structure equations of the Lichnerowicz connection are
\begin{eqnarray}
&&d\theta^i=-\varphi_j^i\wedge\theta^j+\tau^i,\\
&&d\varphi_j^i=-\varphi_k^i\wedge\varphi_j^k+\Phi_j^i,
\end{eqnarray}
where $\tau^i=-\mu_j^i\wedge \overline{\theta^j}$ is the torsion form, and $\Phi_j^i=-\mu_k^i\wedge\overline{\mu_j^k}+\frac{1}{2}(\Omega_j^i+\Omega_{n+j}^{n+i})+\frac{\sqrt{-1}}{2}(\Omega_j^{n+i}-\Omega_{n+j}^i)$
is the curvature form.

Hence, using the skew-symmetric property of $\Omega$ and $\mu$, the first Chern form associated to the Lichnerowicz connection, denoted by $\rho_1(0)$, is
\begin{equation}
\rho_1(0)=\sqrt{-1}\mathrm{tr}(\Phi)=\sqrt{-1}\mu_j^i\wedge\overline{\mu_j^i}-\Omega_j^{n+j}.
\end{equation}
In fact, the two forms $\sqrt{-1}\mu_j^i\wedge\overline{\mu_j^i}$ and $-\Omega_j^{n+j}$ are globally defined. In a more familiar version, these two forms can
be rewritten as follows \cite{SeV},
$$-\Omega_j^{n+j}=\mathfrak{R}(F)=\rho_J,$$
$$\sqrt{-1}\mu_j^i\wedge\overline{\mu_j^i}(X,Y)=\frac{1}{4}\langle J(\nabla_X J)e_A,(\nabla_Y J)e_A\rangle,$$
where $X,Y\in \Gamma(TM).$

In particular, if $(M,J,h)$ is a Hermitian manifold, equivalently, $\nabla_{JX}J=J\nabla_X J$, then $\sqrt{-1}\mu_j^i\wedge\overline{\mu_j^i}$ is a
non-negative $(1,1)$-form.

As a classical example, we consider the six sphere $S^6$. For the standard round metric,
the associated fundamental 2-form $F$ of an orthogonal almost complex structure $J$ satisfies
$\mathfrak{R}(F)=F$. If $J$ is integrable, thus from (3.18), $\rho_1(0)$ is a closed positive  2-form. It is a contradiction since the second betti number
 of $S^6$ is zero. Moreover, we can get the following well-known result.

\begin{prop} \cite{LeB,SeV,BH1,PT,Tang,LY}
The six sphere $S^6$ does not admit an orthogonal complex structure compatible with any metric in a small neighborhood of the standard round metric.
\end{prop}

\section{Curvatures of the canonical Hermitian connections}

In this section, on an almost Hermitian manifold, by using the structure equations of the Lichnerowicz connection $D^0$
and the Chern connection $D^1$, we obtain the corresponding curvature formulas of Gauduchon's family of canonical Hermitian connections $D^t$.
In particular, we show the explicit formulas of two Hermitian scalar curvatures in terms of Riemannian scalar curvature, norms of decompositions of covariant derivative of the fundamental 2-form with respect to the Levi-Civita connection, and the codifferential of the Lee form.

For an almost Hermitian manifold $(M,J,h)$, we consider local unitary frame field $u_i$, and its dual coframe field $\theta^i$ as in section 3. Locally,
let $\psi=(\psi_j^i)$ be the matrix of the Chern connection form. Then the structure equations of the Chern connection are given by
\begin{eqnarray}
&&d\theta^i=-\psi_j^i\wedge\theta^j+T^i,\\
&&d\psi_j^i=-\psi_k^i\wedge\psi_j^k+\Psi_j^i,
\end{eqnarray}
where $T^i$ is the torsion form, $\Psi_j^i$ is the curvature form. For the $(1,1)$-component of the torsion form $T^i$ is vanishing, then $T^i$ can be written as $$T^i=\frac{1}{2}T_{jk}^i \theta^j\wedge\theta^k+\frac{1}{2}T_{\bar{j}\bar{k}}^i \overline{\theta^j}\wedge\overline{\theta^k},$$
where $T_{jk}^i+T_{kj}^i=0$, $T_{\bar{j}\bar{k}}^i+T_{\bar{k}\bar{j}}^i=0$.

Next, we  rewrite the Nijenhuis tensor $N$ and components of $dF$ in terms of the torsion of the Chern connection.

Set $N_{\bar{i}\bar{j}}^k=\langle N(u_{\bar{i}},u_{\bar{j}}),u_{\bar{k}}\rangle$. From the definition of the Nijenhuis tensor $N$, we have
\begin{equation}
N=\frac{1}{2}N_{\bar{i}\bar{j}}^k u_k\otimes \overline{\theta^i}\wedge\overline{\theta^j}+
\frac{1}{2}\overline{N_{\bar{i}\bar{j}}^k} u_{\bar{k}}\otimes \theta^i\wedge\theta^j.
\end{equation}
The torsion tensor $T^{D^1}$ of the Chern connection is defined by
$$T^{D^1}(X,Y)=D_X^1 Y-D_Y^1 X-[X,Y],$$
then
\begin{equation}
T^{D^1}=u_i\otimes T^i+ u_{\bar{i}}\otimes\overline{ T^i}.
\end{equation}
Moreover, since
$$N(X,Y)=-2T^{D^1}(X,Y)-2JT^{D^1}(JX,Y),$$
then (4.3) and (4.4) yield
\begin{equation}
N_{\bar{i}\bar{j}}^k=2\langle-T^{D^1}(u_{\bar{i}},u_{\bar{j}})+\sqrt{-1}JT^{D^1}(u_{\bar{i}},u_{\bar{j}}),u_{\bar{k}}\rangle=-4T_{\bar{i}\bar{j}}^k
\end{equation}

From the structure equation (4.1), the exterior differential of the fundamental 2-form $F$ is given by
$$dF=\sqrt{-1}(T^i\wedge \overline{\theta^i}-\theta^i\wedge\overline{T^i}).$$
Then we have
\begin{equation}
\alpha_F=J\delta F=T_{ji}^i\theta^j+\overline{T_{ji}^i}\overline{\theta^j},
\end{equation}
\begin{equation}
(dF)^+=\frac{\sqrt{-1}}{2}(T_{jk}^i\theta^j\wedge\theta^k\wedge\overline{\theta^i}
-\overline{T_{jk}^i}\overline{\theta^j}\wedge\overline{\theta^k}\wedge\theta^i),
\end{equation}
\begin{equation}
(dF)^-=\frac{\sqrt{-1}}{2}(T_{\bar{j}\bar{k}}^i\overline{\theta^j}\wedge\overline{\theta^k}\wedge\overline{\theta^i}
-\overline{T_{\bar{j}\bar{k}}^i}\theta^j\wedge\theta^k\wedge\theta^i).
\end{equation}
From (4.4)--(4.8), some direct calculations show the following formulas of norms,
\begin{equation}
|T^{D^1}|^2=T_{jk}^i\overline{T_{jk}^i}+T_{\bar{j}\bar{k}}^i\overline{T_{\bar{j}\bar{k}}^i},
\end{equation}
\begin{equation}
|N|^2=16T_{\bar{j}\bar{k}}^i\overline{T_{\bar{j}\bar{k}}^i},~~~~|\alpha_F|^2=2T_{ji}^i\overline{T_{jk}^k},
\end{equation}
\begin{equation}
|(dF)^+|^2=T_{jk}^i\overline{T_{jk}^i},~~~~|(dF)^-|^2=T_{\bar{j}\bar{k}}^i\overline{T_{\bar{j}\bar{k}}^i}+2T_{\bar{j}\bar{k}}^i\overline{T_{\bar{i}\bar{j}}^k}.
\end{equation}

Combing the structure equations (3.16) and (4.1), we have
\begin{equation}
(\varphi_j^i-\psi_j^i)\wedge \theta^j+T^i-\tau^i=0.
\end{equation}
From the following facts,
$$\varphi_j^i+\overline{\varphi_i^j}=0,~~\psi_j^i+\overline{\psi_i^j}=0,~~\mu_j^i+\mu_i^j=0,$$
and some analysis of the form types in (4.12), it follows \cite{Kob2}
\begin{equation}
\varphi_j^i-\psi_j^i=\frac{1}{2}T_{jk}^i \theta^k-\frac{1}{2}\overline{T_{ik}^j} \overline{\theta^k},
\end{equation}
\begin{equation}
\mu_j^i=-\frac{1}{2}\overline{T_{ij}^k} \theta^k+\frac{1}{2}(T_{\bar{j}\bar{k}}^i+T_{\bar{k}\bar{i}}^j-T_{\bar{i}\bar{j}}^k)\overline{\theta^k}.
\end{equation}

Set $\gamma=(\gamma_j^i)=(\varphi_j^i-\psi_j^i)$. Locally, let $\psi(t)=(\psi_j^i(t))$ be the connection form matrix of the canonical Hermitian connection $D^t$. Then
$\psi(t)=\varphi-t\gamma$. In particular, $\psi(0)=\varphi$ and $\psi(1)=\psi$, corresponding to the Lichnerowicz connection $D^0$
and the Chern connection $D^1$ ,respectively.

The structure equations of the canonical Hermitian connection $D^t$ are
\begin{eqnarray}
&&d\theta^i=-\psi_j^i(t)\wedge\theta^j+T^i(t),\\
&&d\psi_j^i(t)=-\psi_k^i(t)\wedge\psi_j^k(t)+\Psi_j^i(t),
\end{eqnarray}
where $T^i(t)$ is the torsion form, $\Psi_j^i(t)$ is the curvature form.
In particular, $T^i(0)=\tau^i$, $\Psi_j^i(0)=\Phi_j^i$, and  $T^i(1)=T^i$, $\Psi_j^i(1)=\Psi_j^i$.

From the above structure equation (4.16), the first Chern form associated to the  canonical Hermitian connection $D^t$, denoted by $\rho_1(t)$, is
\begin{align}\label{E:1}
\rho_1(t)&=\sqrt{-1}\mathrm{tr}(\Psi(t))=\sqrt{-1}\Phi_j^j-\sqrt{-1}t ~d \gamma_j^j\notag\\
&=\sqrt{-1}\Phi_j^j-\sqrt{-1}\frac{t}{2} ~d (T_{jk}^j\theta^k-\overline{T_{jk}^j}\overline{\theta^k})\notag\\
&=\rho_1(0)+\frac{t}{2}~d \delta F.
\end{align}
The above formula is also obtained by using different methods in \cite{Gau3,DGM}.

Extending the action of the almost complex structure $J$ to any $k$-form $\phi$ by
$$(J\phi)(X_1,X_2,...,X_k)=(-1)^k \phi(JX_1,JX_2,...,JX_k),$$
where $X_1,X_2,...,X_k\in\Gamma(TM)$. Set  $\rho^{(1)}(t)=\frac{1}{2}[\rho_1(t)+J\rho_1(t)]$, i.e., $\rho^{(1)}(t)$ is the $(1,1)$-component of the  first Chern form $\rho_1(t)$. If $(M,J,h)$ is
a Hermitian manifold, it is well-known that the  first Chern form $\rho_1(1)$ is a real $(1,1)$-form. Then in the integrable case, (4.17) implies
\begin{equation}
\rho^{(1)}(1)=\rho_1(1)=\rho_1(0)+\frac{1}{2}~d \delta F.
\end{equation}
Hence, we get the following generalization of Theorem 1.2 in \cite{LY},
\begin{prop}
Let $(M,J,h)$ be a compact Hermitian manifold. $c_1^{AC}(M)$ is the first Aeppli-Chern class of the anti-canonical line bundle $K_M^{\ast}$. Then
$\rho^{(1)}(t)$ represents $c_1^{AC}(M)$ in $H_A^{1,1}(M;\mathrm{R})$. More precisely,
$$\rho^{(1)}(t)=\rho_1(1)+\frac{t-1}{2}(\partial\partial^\ast F+\bar{\partial}\bar{\partial}^\ast F).$$
\end{prop}

\begin{proof}
For a Hermitian manifold, from (4.17) and (4.18), we have
\begin{align*}\label{E:1}
\rho^{(1)}(t)&=\frac{1}{2}[\rho_1(0)+\frac{t}{2}~d \delta F+J(\rho_1(0)+\frac{t}{2}~d \delta F)]\notag\\
&=\frac{1}{2}[\rho_1(1)+\frac{t-1}{2}~d \delta F+\rho_1(1)+\frac{t-1}{2}J ~d \delta F)]\notag\\
&=\rho_1(1)+\frac{t-1}{4}(d \delta F+J~d \delta F).
\end{align*}
Since on a compact Hermitian manifold, $d=\partial+\bar{\partial}$, $\delta=-\ast d \ast=\partial^\ast+\bar{\partial}^\ast$,
where $\partial^\ast=-\ast \bar{\partial} \ast$ and $\bar{\partial}^\ast=-\ast\partial\ast$ are formal adjoints of $\partial$ and $\bar{\partial}$, respectively. Then
$$d \delta F+J~d \delta F=2(\partial\partial^\ast F+\bar{\partial}\bar{\partial}^\ast F).$$
Now the result follows.

\end{proof}

As in section 3, we define the curvature tensor, denoted by $K^t$, of the canonical Hermitian connection $D^t$ as follows,
$$
K^t(X,Y,Z,W)=\langle D_Z^t D_W^t Y-D_W^t D_Z^t Y-D_{[Z,W]}^t Y, X\rangle,
$$
where $X,Y,Z,W\in \Gamma(TM)$.
Then the curvature form
\begin{equation}
\Psi_j^i(t)=\frac{1}{2}K_{\bar{i}jkl}^t\theta^k\wedge\theta^l+K_{\bar{i}jk\bar{l}}^t\theta^k\wedge\overline{\theta^l}
+\frac{1}{2}K_{\bar{i}j\bar{k}\bar{l}}^t\overline{\theta^k}\wedge\overline{\theta^l},
\end{equation}
where $K_{\bar{i}jkl}^t+K_{\bar{i}jlk}^t=0$, $K_{\bar{i}j\bar{k}\bar{l}}^t+K_{\bar{i}j\bar{l}\bar{k}}^t=0$.
Here, $K_{\bar{i}jkl}^t=K^t(u_{\bar{i}},u_j,u_k,u_l)$, others are similar.

\begin{prop}
Let $(M,J,h)$ be an almost Hermitian manifold of real dimension $2n$. Then the components of the curvature tensor  $K^t$ of the canonical Hermitian connection $D^t$ are given by
\begin{eqnarray}
&&K_{\bar{i}jkl}^t=R_{\bar{i}jkl}
+\frac{1}{4}[\overline{T_{ip}^k}(\overline{T_{\bar{j}\bar{l}}^p}+\overline{T_{\bar{l}\bar{p}}^j}-\overline{T_{\bar{p}\bar{j}}^l})
-\overline{T_{ip}^l}(\overline{T_{\bar{j}\bar{k}}^p}+\overline{T_{\bar{k}\bar{p}}^j}-\overline{T_{\bar{p}\bar{j}}^k})]\\
&&~~~~~~~~~~~~~-\frac{t}{2}(T_{jl,k}^i-T_{jk,l}^i+T_{jp}^i T_{kl}^p-\overline{T_{ip}^j} \overline{T_{\bar{k}\bar{l}}^p})
+\frac{t^2-2t}{4}(T_{pk}^i T_{jl}^p-T_{pl}^i T_{jk}^p)\notag,\\
&&K_{\bar{i}jk\bar{l}}^t=R_{\bar{i}jk\bar{l}}
+\frac{1}{4}(T_{\bar{p}\bar{l}}^i+T_{\bar{l}\bar{i}}^p-T_{\bar{i}\bar{p}}^l)
(\overline{T_{\bar{j}\bar{k}}^p}+\overline{T_{\bar{k}\bar{p}}^j}-\overline{T_{\bar{p}\bar{j}}^k})-\frac{1}{4}\overline{T_{ip}^k} T_{pj}^l\\
&&~~~~~~~~~~~~~+\frac{t}{2}(T_{jk,\bar{l}}^i+\overline{T_{il,\bar{k}}^j})
+\frac{t^2-2t}{4}( T_{jk}^p \overline{T_{il}^p}-T_{pk}^i \overline{T_{pl}^j})\notag,\\
&&K_{\bar{i}j\bar{k}\bar{l}}^t=R_{\bar{i}j\bar{k}\bar{l}}
+\frac{1}{4}[T_{pj}^l(T_{\bar{p}\bar{k}}^i+T_{\bar{k}\bar{i}}^p-T_{\bar{i}\bar{p}}^k)
-T_{pj}^k(T_{\bar{p}\bar{l}}^i+T_{\bar{l}\bar{i}}^p-T_{\bar{i}\bar{p}}^l)]\\
&&~~~~~~~~~~~~~-\frac{t}{2}(\overline{T_{ik,l}^j}-\overline{T_{il,k}^j}+T_{jp}^i T_{\bar{k}\bar{l}}^p-\overline{T_{ip}^j} \overline{T_{kl}^p})
+\frac{t^2-2t}{4}(\overline{T_{ik}^p} \overline{T_{pl}^j}-\overline{T_{il}^p} \overline{T_{pk}^j})\notag,
\end{eqnarray}
where $T_{jk,l}^i$ and $T_{jk,\bar{l}}^i$ are defined by
\begin{equation}
dT_{jk}^i+T_{jk}^p \psi_p^i-T_{pk}^i \psi_j^p-T_{jp}^i \psi_k^p=T_{jk,l}^i \theta^l+T_{jk,\bar{l}}^i\overline{\theta^l}.
\end{equation}
\end{prop}

\begin{proof}
From the structure equation (4.16), the curvature form $\Psi_j^i(t)$ is
\begin{align}\label{E:1}
\Psi_j^i(t)&=\Psi_j^i(0)-t(d\gamma_j^i+\varphi_k^i\wedge\gamma_j^k+\gamma_k^i\wedge\varphi_j^k)+t^2\gamma_k^i\wedge\gamma_j^k\notag\\
&=\Psi_j^i(0)-t(d\gamma_j^i+\psi_k^i\wedge\gamma_j^k+\gamma_k^i\wedge\psi_j^k)+(t^2-2t)\gamma_k^i\wedge\gamma_j^k,
\end{align}
where
$$\Psi_j^i(0)=\Phi_j^i=-\mu_k^i\wedge\overline{\mu_j^k}+\frac{1}{2}(\Omega_j^i+\Omega_{n+j}^{n+i})+\frac{\sqrt{-1}}{2}(\Omega_j^{n+i}-\Omega_{n+j}^i)$$
is the curvature form of the Lichnerowicz connection $D^0$.

From (4.14), we get the form decomposition of $\mu_p^i\wedge\overline{\mu_j^p}$ as follows,
\begin{eqnarray}
&&(\mu_p^i\wedge\overline{\mu_j^p})^{(2,0)}=-\frac{1}{4}\overline{T_{ip}^k}(\overline{T_{\bar{j}\bar{l}}^p}
+\overline{T_{\bar{l}\bar{p}}^j}-\overline{T_{\bar{p}\bar{j}}^l})\theta^k\wedge\theta^l,\\
&&(\mu_p^i\wedge\overline{\mu_j^p})^{(0,2)}=
-\frac{1}{4}T_{pj}^l(T_{\bar{p}\bar{k}}^i+T_{\bar{k}\bar{i}}^p-T_{\bar{i}\bar{p}}^k)\overline{\theta^k}\wedge\overline{\theta^l},\\
&&(\mu_p^i\wedge\overline{\mu_j^p})^{(1,1)}=\frac{1}{4}[\overline{T_{ip}^k} T_{pj}^l-(T_{\bar{p}\bar{l}}^i
+T_{\bar{l}\bar{i}}^p-T_{\bar{i}\bar{p}}^l)
(\overline{T_{\bar{j}\bar{k}}^p}+\overline{T_{\bar{k}\bar{p}}^j}-\overline{T_{\bar{p}\bar{j}}^k})]\theta^k\wedge\overline{\theta^l}.
\end{eqnarray}
For the Riemannian curvature component, since $$\frac{1}{2}(\Omega_j^i+\Omega_{n+j}^{n+i})+\frac{\sqrt{-1}}{2}(\Omega_j^{n+i}-\Omega_{n+j}^i)=R(u_{\bar{i}},u_j,\cdot,\cdot),$$
then we have
\begin{equation}
\frac{1}{2}(\Omega_j^i+\Omega_{n+j}^{n+i})+\frac{\sqrt{-1}}{2}(\Omega_j^{n+i}-\Omega_{n+j}^i)=\frac{1}{2}R_{\bar{i}jkl}\theta^k\wedge\theta^l
+R_{\bar{i}jk\bar{l}}\theta^k\wedge\overline{\theta^l}+\frac{1}{2}R_{\bar{i}j\bar{k}\bar{l}}\overline{\theta^k}\wedge\overline{\theta^l}.
\end{equation}

From (4.13), it follows
\begin{equation}
\gamma_p^i\wedge\gamma_j^p=\frac{1}{4}T_{pk}^i T_{jl}^p\theta^k\wedge\theta^l
+\frac{1}{4}\overline{T_{ik}^p} \overline{T_{pl}^j}\overline{\theta^k}\wedge\overline{\theta^l}+\frac{1}{4}( T_{jk}^p \overline{T_{il}^p}-T_{pk}^i \overline{T_{pl}^j})\theta^k\wedge\overline{\theta^l}.
\end{equation}
Combing (4.13) and the structure equation (4.1), a direct calculation shows that
\begin{eqnarray}
&&d\gamma_j^i+\psi_p^i\wedge\gamma_j^p+\gamma_p^i\wedge\psi_j^p=
\frac{1}{2}(dT_{jl}^i+T_{jl}^p \psi_p^i-T_{pl}^i \psi_j^p-T_{jp}^i \psi_l^p)\wedge\theta^l\notag\\
&&~~~~~~~+\frac{1}{2}T_{jp}^i T^p-\frac{1}{2}\overline{T_{ip}^j} \overline{T^p}-\frac{1}{2}(d\overline{T_{il}^j}+\overline{T_{il}^p} \overline{\psi_p^j}
-\overline{T_{pl}^j} \overline{\psi_i^p}-\overline{T_{ip}^j} \overline{\psi_l^p})\wedge\overline{\theta^l}.
\end{eqnarray}
Then from (4.23),  (4.30) implies
\begin{eqnarray}
&&d\gamma_j^i+\psi_p^i\wedge\gamma_j^p+\gamma_p^i\wedge\psi_j^p=
\frac{1}{2}(T_{jl,k}^i\theta^k+T_{jl,\bar{k}}^i\overline{\theta^k})\wedge\theta^l\notag\\
&&~~~~~~~+\frac{1}{2}T_{jp}^i T^p-\frac{1}{2}\overline{T_{ip}^j} \overline{T^p}
-\frac{1}{2}(\overline{T_{il,k}^j}\overline{\theta^k}+\overline{T_{il,\bar{k}}^j}\theta^k)\wedge\overline{\theta^l}.
\end{eqnarray}
Put (4.25)--(4.31) into the formula (4.24), then we get the results.
\end{proof}

As done in section 3, by using contractions of  the curvature component $K_{\bar{i}jk\bar{l}}^t$,  two Hermitian scalar curvatures $s_1(t)$ and $s_2(t)$ of the canonical Hermitian
connection $D^t$ are defined by
\begin{equation}
s_1(t)=K_{\bar{i}ij\bar{j}}^t,~~~~s_2(t)=K_{\bar{i}ji\bar{j}}^t.
\end{equation}

\begin{thm}
Let $(M,J,h)$ be an almost Hermitian manifold of real dimension $2n$.
Then
\begin{equation}
\begin{aligned}
s_1(t)
=&\frac{s}{2}-\frac{5}{12}|(dF)^-|^2+\frac{1}{16}|N^0|^2+\frac{1}{4}|(dF)_0^+|^2\\
&+[\frac{1}{4(n-1)}
+\frac{t-1}{2}]|\alpha_F|^2+\frac{t-2}{2}\delta \alpha_F
\end{aligned}
\end{equation}
and
\begin{equation}
\begin{aligned}
s_2(t)
=&\frac{s}{2}-\frac{1}{12}|(dF)^-|^2+\frac{1}{32}|N^0|^2-\frac{t^2-2t}{4}|(dF)_0^+|^2\\
&-[\frac{t^2-2t}{4(n-1)}+\frac{(t+1)^2}{8}]|\alpha_F|^2-\frac{t+1}{2}\delta \alpha_F.
\end{aligned}
\end{equation}
\end{thm}

\begin{proof}
For $s_1(t)$, from (4.22), we have
\begin{align}\label{E:1}
s_1(t)&=s_1(0)+\frac{t}{2}(T_{ij,\bar{j}}^i+\overline{T_{ij,\bar{j}}^i})\notag\\
&=s_1(0)+\frac{t}{2}\langle d\delta F, F\rangle=s_1(0)+\frac{t}{2}(|\alpha_F|^2+\delta \alpha_F).
\end{align}
Then  (3.12) and  (4.35) imply (4.33).

For $s_2(t)$, from (4.22), we have
\begin{align}\label{E:1}
s_2(t)&=s_2(0)+\frac{t}{2}(T_{ji,\bar{j}}^i+\overline{T_{ij,\bar{i}}^j})
+\frac{t^2-2t}{4}( T_{ji}^p \overline{T_{ij}^p}-T_{pi}^i \overline{T_{pj}^j})\notag\\
&=s_2(0)-\frac{t}{2}(|\alpha_F|^2+\delta \alpha_F)
-\frac{t^2-2t}{4}( T_{ij}^p \overline{T_{ij}^p}+T_{pi}^i \overline{T_{pj}^j}).
\end{align}
Combing with (4.10) and (4.11), it follows that
\begin{equation}
s_2(t)=s_2(0)-\frac{t}{2}(|\alpha_F|^2+\delta \alpha_F)
-\frac{t^2-2t}{4}|(dF)^+|^2-\frac{t^2-2t}{8}|\alpha_F|^2.
\end{equation}
Then  (3.13) and  (4.37) imply (4.34).

\end{proof}

Now, we give two particular cases of Theorem 4.3.  For Hermitian manifolds, equivalently, $(dF)^-=N^0=0$, we have

\begin{cor}
Let $(M,J,h)$ be a Hermitian manifold of real dimension $2n$.  Then
\begin{equation}
s_1(t)=\frac{s}{2}+\frac{1}{4}|dF|^2+\frac{t-1}{2}|\alpha_F|^2+\frac{t-2}{2}\delta \alpha_F
\end{equation}
and
\begin{equation}
s_2(t)=\frac{s}{2}-\frac{t^2-2t}{4}|dF|^2-\frac{(t+1)^2}{8}|\alpha_F|^2-\frac{t+1}{2}\delta \alpha_F.
\end{equation}
\end{cor}

For almost Hermitian surfaces, $(dF)^-=(dF)_{0}^+=0$,  then we have

\begin{cor}
Let $(M,J,h)$ be an almost Hermitian surface. Then
\begin{equation}
s_1(t)=\frac{s}{2}+\frac{1}{16}|N|^2+\frac{2t-1}{4}|\alpha_F|^2+\frac{t-2}{2}\delta \alpha_F
\end{equation}
and
\begin{equation}
s_2(t)=\frac{s}{2}+\frac{1}{32}|N|^2-\frac{3t^2-2t+1}{8}|\alpha_F|^2-\frac{t+1}{2}\delta \alpha_F.
\end{equation}
\end{cor}

\begin{rem}
From Theorem 4.3, we can obtain explicit relations among various scalar curvatures, for each of the 16 classes
( 4 classes for $n=2$) of almost Hermitian manifolds given by the Gray-Hervella classification.
For Hermitian manifolds, there are many studies with respect to the Hermitian scalar curvatures of the
 Bismut connection, the Lichnerowicz connection and the Chern connection (i.e., corresponding to $t=-1,0,1$ ) \cite{Gau1,AlI,IP1,LY,FuZ}. Recently, for almost Hermitian manifolds, $s_1(1)$ and $s_2(1)$  are also given in \cite{LeU} by using different method.
\end{rem}

\begin{rem}
As in the Hermitian case \cite{Gau1,LY}, on an almost Hermitian manifold, one can also define
four Ricci forms, denoted by $\rho^{(1)}(t)$, $\rho^{(2)}(t)$, $\rho^{(3)}(t)$ and $\rho^{(4)}(t)$, of canonical Hermitian connection $D^t$ as
follows,
$$\rho^{(1)}(t)=\sqrt{-1}K_{\bar{k}ki\bar{j}}^t\theta^i\wedge\overline{\theta^j},
~~~~\rho^{(2)}(t)=\sqrt{-1}K_{\bar{j}ik\bar{k}}^t\theta^i\wedge\overline{\theta^j},$$
$$\rho^{(3)}(t)=\sqrt{-1}K_{\bar{k}ik\bar{j}}^t\theta^i\wedge\overline{\theta^j},
~~~~\rho^{(4)}(t)=\sqrt{-1}K_{\bar{j}ki\bar{k}}^t\theta^i\wedge\overline{\theta^j}.$$
These Ricci forms are very useful in the research of almost Hermitian curvature flows and cohomology groups of almost Hermitian manifold, which
are worthy to study furthermore. One can refer to some related works in \cite{LY,ST1,ST2,ChZ}.

\end{rem}

\section{Some applications}

In this section, we show some applications of Theorem 4.3.

On a compact Hermitian manifold $(M,J,h)$ with the Chern connection, (4.38) implies the following integral formula of the first
Hermitian scalar curvature $s_1(1)$ and the Riemannian scalar curvature $s$,
$$\int_M [2s_1(1)-s]dv=\frac{1}{2}\int_M |dF|^2 dv\geq 0.$$
The equality holds if and only if $dF=0$, i.e., $(M,J,h)$ is a K\"{a}hler manifold \cite{Gau1,LY}.

In general, for an almost Hermitian manifold $(M,J,h)$ with canonical Hermitian connection $D^t$, we have

\begin{thm}
Let $(M,J,h)$ be a compact almost Hermitian manifold of real dimension $2n$ ($n\geq3$).

(1) If $(M,J,h)\in \mathcal{W}_2\oplus \mathcal{W}_3\oplus \mathcal{W}_4$
and $t\geq 1-\frac{1}{2(n-1)}$, then
\begin{equation*}
\int_M [2s_1(t)-s]dv\geq 0.
\end{equation*}
The equality holds if and only if $(M,J,h)$ is a locally conformally  K\"{a}hler manifold when $t=1-\frac{1}{2(n-1)}$ or
$(M,J,h)$ is a K\"{a}hler manifold when $t>1-\frac{1}{2(n-1)}$.

(2) If $(M,J,h)\in \mathcal{W}_1\oplus \mathcal{W}_4$
and $t\leq 1-\frac{1}{2(n-1)}$, then
$$\int_M [2s_1(t)-s]dv\leq 0.$$
The equality holds  if and only if $(M,J,h)$ is a locally conformally  K\"{a}hler manifold when $t=1-\frac{1}{2(n-1)}$ or
$(M,J,h)$ is a K\"{a}hler manifold when $t<1-\frac{1}{2(n-1)}$.

\end{thm}

\begin{proof}
(1) If $(M,J,h)\in \mathcal{W}_2\oplus \mathcal{W}_3\oplus \mathcal{W}_4$, then
\begin{equation}
(dF)^-=0,~~~~N=N^0.
\end{equation}
Combing (4.33) and (5.1), for $t\geq 1-\frac{1}{2(n-1)}$, we have
$$\int_M[2s_1(t)-s]dv=\frac{1}{8}\int_M |N|^2dv+\frac{1}{2}\int_M |(dF)_0^+|^2dv
+[t-1+\frac{1}{2(n-1)}]\int_M|\alpha_F|^2dv\geq 0.$$

If $\int_M[2s_1(t)-s]dv=0$, then three components in the above formula are all zero. The results followed.

(2) If $(M,J,h)\in \mathcal{W}_1\oplus \mathcal{W}_4$, then
\begin{equation}
N^0=0,~~~~(dF)_0^+=0.
\end{equation}
Combing (4.33) and (5.2), for $t\leq 1-\frac{1}{2(n-1)}$, we have
$$\int_M[2s_1(t)-s]dv=-\frac{5}{6}\int_M |(dF)^-|^2dv
+[t-1+\frac{1}{2(n-1)}]\int_M|\alpha_F|^2dv\leq 0.$$

If $\int_M[2s_1(t)-s]dv=0$, then two components in the above formula are all zero. The results followed.
\end{proof}

From Corollary 4.5, for almost Hermitian surfaces,  by using similar methods as in Theorem 5.1, we get
\begin{thm}
Let $(M,J,h)$ be a compact almost Hermitian surface.

(1) If $t\geq \frac{1}{2}$, then
$$\int_M [2s_1(t)-s]dv\geq 0.$$
The equality holds if and only if $(M,J,h)$ is a Hermitian surface when $t=\frac{1}{2}$ or  $(M,J,h)$ is a K\"{a}hler surface when $t>\frac{1}{2}$.

(2) If $J$ is integrable
and $t\leq \frac{1}{2}$, then
$$\int_M [2s_1(t)-s]dv\leq 0.$$
The equality holds if and only if $t=\frac{1}{2}$  or $(M,J,h)$ is a K\"{a}hler surface when $t<\frac{1}{2}$.

\end{thm}

Recently, for the Chern connection $D^1$, M. G. Dabkowski and M. T. Lock \cite{DaL} have constructed non-compact Hermitian manifolds with $2s_1(1)=s$ which are not K\"{a}hler manifolds. M. Lejmi and  M. Upmeier have given a problem in  Remark 3.3 in \cite{LeU}:  do higher-dimensional compact almost Hermitian non-K\"{a}hler manifolds with $2s_1(1)=s$ exist? For this question, from Theorem 5.1 (1), we have the following non-existence result.

\begin{cor}
Let $(M,J,h)$ be a compact almost Hermitian manifold of real dimension $2n$ ($n\geq3$).  If $(M,J,h)\in \mathcal{W}_2\oplus \mathcal{W}_3
\oplus \mathcal{W}_4$ and $2s_1(1)=s$, then $(M,J,h)$ is a K\"{a}hler manifold.

\end{cor}

On the other hand, for a compact Hermitian manifold $(M,J,h)$, from (4.38) and (4.39), the two Hermitian scalar curvatures $s_1(1)$ and $s_2(1)$ of the Chern connection satisfy
$$s_1(1)-s_2(1)=\frac{1}{2}|\alpha_F|^2+\frac{1}{2}\delta\alpha_F,$$
which implies the following well-known integral result \cite{Gau1},
\begin{equation}
\int_M [s_1(1)-s_2(1)]dv=\frac{1}{2}\int_M |\alpha_F|^2dv\geq 0.
\end{equation}
The equality holds if and only if $\alpha_F=0$, i.e., $(M,J,h)$ is a balanced manifold.

In fact, we can obtain similar integral formula on a compact almost Hermitian manifold $(M,J,h)$. From Theorem 4.3,
the difference between $s_1(t)$ and $s_2(t)$  of the canonical Hermitian connection $D^t$ is
\begin{eqnarray}
&&s_1(t)-s_2(t)=-\frac{1}{3}|(dF)^-|^2+\frac{1}{32}|N^0|^2+\frac{(t-1)^2}{4}|(dF)_0^+|^2\notag\\
&&~~~~~~~~~~~~~~~~~~~~~~~~~+\frac{(n+1)t^2+(6n-10)t+5-3n}{8(n-1)}|\alpha_F|^2+(t-\frac{1}{2})\delta\alpha_F.\notag
\end{eqnarray}
Then we get a general integral formula
\begin{eqnarray}
&&\int_M[s_1(t)-s_2(t)]dv=-\frac{1}{3}\int_M |(dF)^-|^2dv+\frac{1}{32}\int_M |N^0|^2dv\\
&&~~~+\frac{(t-1)^2}{4}\int_M|(dF)_0^+|^2dv+\frac{(n+1)t^2+(6n-10)t+5-3n}{8(n-1)}\int_M|\alpha_F|^2dv.\notag
\end{eqnarray}

\begin{thm}
Let $(M,J,h)$ be a compact almost Hermitian manifold of real dimension $2n$ ($n\geq 3$).

(1) If $(M,J,h)\in \mathcal{W}_2\oplus \mathcal{W}_3\oplus \mathcal{W}_4$, $t\in (-\infty, -3-2\sqrt{3}]\cup [-3+2\sqrt{3},+\infty)$, then
$$\int_M[s_1(t)-s_2(t)]dv\geq 0.$$
The equality holds if and only if $(M,J,h)$ is a balanced Hermitian manifold when $t=1$ or  $(M,J,h)$ is a K\"{a}hler manifold when $t\neq 1$.

(2) If $(M,J,h)\in \mathcal{W}_1\oplus \mathcal{W}_4$, $t\in [-1,\frac{1}{3}]$, then
$$\int_M[s_1(t)-s_2(t)]dv\leq 0.$$
The equality holds if and only if  $(M,J,h)$ is a K\"{a}hler manifold.

\end{thm}

\begin{proof}
(1) Obviously, $t\in (-\infty, \frac{5-3n-2\sqrt{3n^2-8n+5}}{n+1}]\cup [\frac{5-3n+2\sqrt{3n^2-8n+5}}{n+1},+\infty)$ yields $(n+1)t^2+(6n-10)t+5-3n\geq 0$.
In fact, the sequence $\{\frac{5-3n-2\sqrt{3n^2-8n+5}}{n+1}\}$ is monotonically decreasing, and the sequence $\{\frac{5-3n+2\sqrt{3n^2-8n+5}}{n+1}\}$ is monotonically increasing. For $n\rightarrow +\infty$, the limits of these two sequences are $-3-2\sqrt{3}$ and $-3+2\sqrt{3}$, respectively. Then for $t\in (-\infty, -3-2\sqrt{3}]\cup [-3+2\sqrt{3},+\infty) $,
$$(n+1)t^2+(6n-10)t+5-3n >0.$$

For $(M,J,h)\in \mathcal{W}_2\oplus \mathcal{W}_3\oplus \mathcal{W}_4$, combing (5.1) and (5.4), we get
\begin{eqnarray}
&&\int_M[s_1(t)-s_2(t)]dv=\frac{1}{32}\int_M |N|^2dv+\frac{(t-1)^2}{4}\int_M|(dF)_0^+|^2dv\notag\\
&&~~~~~~~~~+\frac{(n+1)t^2+(6n-10)t+5-3n}{8(n-1)}\int_M|\alpha_F|^2dv\geq 0.\notag
\end{eqnarray}
If $\int_M[s_1(t)-s_2(t)]dv=0$, then three components in the above formula are all zero. A direct analysis implies the results.

(2) For $t\in [-1,\frac{1}{3}]$, $(n+1)t^2+(6n-10)t+5-3n <0$. Since $(M,J,h)\in \mathcal{W}_1\oplus \mathcal{W}_4$, from (5.2) and (5.4), we have
$$\int_M[s_1(t)-s_2(t)]dv=\frac{(n+1)t^2+(6n-10)t+5-3n}{8(n-1)}\int_M|\alpha_F|^2dv-\frac{1}{3}\int_M |(dF)^-|^2dv.$$
Then $\int_M[s_1(t)-s_2(t)]dv\leq0$, with equality holding if and only if  $(dF)^-=\alpha_F=0$, i.e., $(M,J,h)$ is a K\"{a}hler manifold.

\end{proof}

For completeness, we also give a similar result for  almost Hermitian surfaces.
\begin{thm}
Let $(M,J,h)$ be a compact almost Hermitian surface.

(1) If  $t\in (-\infty, -1]\cup [\frac{1}{3},+\infty)$, then
$$\int_M[s_1(t)-s_2(t)]dv\geq 0.$$
The equality holds if and only if $(M,J,h)$ is a Hermitian surface when $t\in\{-1, \frac{1}{3}\}$ or $(M,J,h)$ is a K\"{a}hler surface when $t\in (-\infty, -1)\cup (\frac{1}{3},+\infty)$.

(2) If $J$ is integrable and $t\in [-1,\frac{1}{3}]$, then
$$\int_M[s_1(t)-s_2(t)]dv\leq 0.$$
The equality holds if and only if $t=-1$, or $t=\frac{1}{3}$ or $(M,J,h)$ is a K\"{a}hler surface when $t\in (-1,\frac{1}{3})$.

\end{thm}
\begin{proof}
From Corollary 4.5, for almost Hermitian surface, we have
$$\int_M[s_1(t)-s_2(t)]dv=\frac{1}{32}\int_M |N|^2dv+\frac{(3t-1)(t+1)}{8}\int_M |\alpha_F|^2dv.$$
Now, the results followed by a direct analysis.

\end{proof}

\begin{rem} There are also many studies of Hermitian manifolds with
flat curvature associated to some canonical Hermitian connection \cite{Boo, WYZ,YZ2}.
Recently, B.  Yang and F. Y. Zheng \cite{YZ1}, and D. Angella, A. Otal, L. Ugarte and R. Villacampa \cite{AOUV} have
introduced the definition of K\"{a}hler-like Hermitian structure on a Hermitian manifold, which is a strong curvature condition.
Obviously, the  K\"{a}hler-like curvature condition in \cite{YZ1,AOUV} implies $\int_M s_1(t)dv=\int_M s_2(t)dv$. The above
Theorem 5.4 and Theorem 5.5 are closely related to conjecture 2 in \cite{AOUV} which states that on a compact
 Hermitian manifold $(M,J,h)$, if it is K\"{a}hler-like for a canonical Hermitian connection $D^t$ with $t\neq 1, -1$, then $(M,J,h)$ is a K\"{a}hler manifold. Our results are obtained under weaker conditions: 1.
$(M,J,h)\in \mathcal{W}_2\oplus \mathcal{W}_3\oplus \mathcal{W}_4$ or $(M,J,h)\in \mathcal{W}_1\oplus \mathcal{W}_4$, not necessary a Hermitian manifold; 2. we only use integrations of the two Hermitian scalar curvatures. Hence, as a corollary, we get
\end{rem}

\begin{cor}
Let $(M,J,h)$ be a compact Hermitian manifold of real dimension $2n$ ($n\geq 3$).  If $(M,J,h)$ is K\"{a}hler-like for a canonical Hermitian connection $D^t$ with $t\in (-\infty, -3-2\sqrt{3}]\cup [-3+2\sqrt{3},1)\cup(1,+\infty)$, then $(M,J,h)$ is a K\"{a}hler manifold.
\end{cor}

Note that $t=-1$ (i.e., the Bismut connection) is not contained in the above corollary. However, for the Bismut connection on a Hermitian manifold, we have

\begin{thm}
Let $(M,J,h)$ be a compact Hermitian manifold of real dimension $2n$ ($n\geq 3$).  Then
$$\int_M[s_1(-1)-s_2(-1)]dv=\int_M(|dF|^2-|\alpha_F|^2)dv.$$
In particular, if $h$ is a Gauduchon metric (i.e., $\delta\alpha_F=0$) and $s_1(-1)=s_2(-1)$, then
$h$ is a $k$-Gauduchon metric \cite{FWW}, that is
$$\sqrt{-1}\partial\bar{\partial} (F^k)\wedge F^{n-k-1}=0,$$ for $k=1,2,...,n-1$.
\end{thm}

\begin{proof}
From (4.38) and (4.39), we have
$$s_2(-1)=\frac{s}{2}-\frac{3}{4}|dF|^2,$$
$$s_1(-1)=\frac{s}{2}+\frac{1}{4}|dF|^2-|\alpha_F|^2-\frac{3}{2}\delta\alpha_F.$$
Then the difference between $s_1(-1)$ and $s_2(-1)$ is

\begin{equation}
s_1(-1)-s_2(-1)=|dF|^2-|\alpha_F|^2-\frac{3}{2}\delta\alpha_F,
\end{equation}
which implies the following integral formula,
$$\int_M[s_1(-1)-s_2(-1)]dv=\int_M(|dF|^2-|\alpha_F|^2)dv.$$

Next, we provide a direct calculation of $\sqrt{-1}\partial\bar{\partial} (F^k)\wedge F^{n-k-1}$ by using the Chern connection.
There are some different methods for dealing with the form $\sqrt{-1}\partial\bar{\partial} (F^k)\wedge F^{n-k-1}$ \cite{IP2,LY}.

Firstly, note that
\begin{equation}
\sqrt{-1}\partial\bar{\partial} (F^k)\wedge F^{n-k-1}
=k(k-1)\sqrt{-1}\partial F\wedge\bar{\partial} F\wedge F^{n-3}+k\sqrt{-1}\partial\bar{\partial} F\wedge F^{n-2}.
\end{equation}
From (4.7), we have
$$\sqrt{-1}\partial F\wedge\bar{\partial} F=-\frac{\sqrt{-1}}{4}T_{jk}^i \overline{T_{pq}^l}
\theta^l\wedge\theta^j\wedge\theta^k\wedge\overline{\theta^i}\wedge\overline{\theta^p}\wedge\overline{\theta^q},$$
which implies
\begin{equation}
\langle\sqrt{-1}\partial F\wedge\bar{\partial} F, F^3\rangle=6T_{ki}^i\overline{T_{kj}^j}-3T_{jk}^i\overline{T_{jk}^i}.
\end{equation}
From (4.7), (4.1) and (4.23),  it follows that
\begin{align}\label{E:1}
\sqrt{-1}\partial\bar{\partial} F&=\frac{1}{2}[d(\overline{T_{jk}^i}\overline{\theta^j}\wedge\overline{\theta^k}\wedge\theta^i)]^{(2,2)}\notag\\
&=(\frac{1}{2}\overline{T_{jk,\bar{p}}^q}
+\frac{1}{4}T_{pq}^i\overline{T_{jk}^i})\theta^p\wedge\theta^q\wedge\overline{\theta^j}\wedge\overline{\theta^k},\notag
\end{align}
which implies

\begin{equation}
\langle\sqrt{-1}\partial\bar{\partial} F, F^2\rangle=T_{jk}^i\overline{T_{jk}^i}+2\overline{T_{ji,\bar{j}}^i}.
\end{equation}

From (4.10) and (4.11), for a Hermitian manifold, we have
\begin{equation}
|\alpha_F|^2=2T_{ki}^i\overline{T_{kj}^j},~~~|dF|^2=T_{jk}^i\overline{T_{jk}^i}.
\end{equation}
Put (5.9) into (5.7) and (5.8), then
\begin{equation}
\langle\sqrt{-1}\partial F\wedge\bar{\partial} F, F^3\rangle
=3(|\alpha_F|^2-|dF|^2),
\end{equation}
\begin{equation}
\langle\sqrt{-1}\partial\bar{\partial} F, F^2\rangle
=|dF|^2-|\alpha_F|^2-\delta\alpha_F.
\end{equation}

Combing (5.6), (5.10) and (5.11), we get
\begin{align}\label{E:1}
\sqrt{-1}\partial\bar{\partial} (F^k)\wedge F^{n-k-1}&=k(k-1)\frac{(n-3)!}{6}\langle\sqrt{-1}\partial F\wedge\bar{\partial} F, F^3\rangle dv\notag\\
&~~~~+k\frac{(n-2)!}{2}\langle\sqrt{-1}\partial\bar{\partial} F, F^2\rangle dv\notag\\
&=k\frac{(n-3)!}{2}[(n-k-1)(|dF|^2-|\alpha_F|^2)-(n-2)\delta\alpha_F]dv.
\end{align}
Furthermore, if $h$ is a Gauduchon metric (i.e., $\delta\alpha_F=0$ ) and $s_1(-1)=s_2(-1)$, then (5.5) implies
$|dF|^2-|\alpha_F|^2=0$.
Hence
$$\sqrt{-1}\partial\bar{\partial} (F^k)\wedge F^{n-k-1}=0,~~~~k=1,2,...,n-1.$$

\end{proof}

\bigskip
\footnotesize
\noindent\textit{Acknowledgments.}
Zhou would like to thank Professor Jiagui Peng  and Professor Xiaoxiang Jiao for their helpful suggestions and encouragements. Part of the work was done while Zhou was visiting Laboratory of Mathematics for Nonlinear Science, Fudan University;
he would like to  thank them for the warm hospitality and supports. Fu is supported in part by NSFC  10831008 and  11025103. Zhou is supported in part by NSFC 11501505.

\end{document}